\documentclass[10 pt]{amsart}
\usepackage{latexsym, amsmath, amssymb, longtable, booktabs,amscd,microtype,booktabs,cases}
\usepackage[square]{natbib}
\usepackage[english]{babel}
\usepackage[utf8x]{inputenc}
\usepackage{hyperref}
\bibpunct{[}{]}{,}{n}{}{;}

\usepackage[english]{babel}
\usepackage[utf8x]{inputenc}
\usepackage{graphicx}
\usepackage{cleveref}
\usepackage{xcolor}

\usepackage[OT2,T1]{fontenc}
\DeclareSymbolFont{cyrletters}{OT2}{wncyr}{m}{n}
\DeclareMathSymbol{\Sha}{\mathalpha}{cyrletters}{"58}
\numberwithin{equation}{section}

\usepackage[switch,columnwise]{lineno}
\usepackage{lipsum} 

\usepackage{fancyhdr,lipsum}
\begin{document}

\newcommand\A{\mathbb{A}}
\newcommand\C{\mathbb{C}}
\newcommand\G{\mathbb{G}}
\newcommand\N{\mathbb{N}}
\newcommand\T{\mathbb{T}}
\newcommand\sO{\mathcal{O}}
\newcommand\sE{{\mathcal{E}}}
\newcommand\tE{{\mathbb{E}}}
\newcommand\sF{{\mathcal{F}}}
\newcommand\sG{{\mathcal{G}}}
\newcommand\sH{{\mathcal{H}}}
\newcommand\sN{{\mathcal{N}}}
\newcommand\GL{{\mathrm{GL}}}
\newcommand\HH{{\mathrm H}}
\newcommand\mM{{\mathrm M}}
\newcommand\fS{\mathfrak{S}}
\newcommand\fP{\mathfrak{P}}
\newcommand\fQ{\mathfrak{Q}}
\newcommand\Qbar{{\bar{\Q}}}
\newcommand\sQ{{\mathcal{Q}}}
\newcommand\sP{{\mathbb{P}}}
\newcommand{\Q}{\mathbb{Q}}
\newcommand{\tH}{\mathbb{H}}
\newcommand{\Z}{\mathbb{Z}}
\newcommand{\R}{\mathbb{R}}
\newcommand{\F}{\mathbb{F}}
\newcommand\cP{\mathcal{P}}
\newcommand\cQ{\mathcal{Q}}
\newcommand\Gal{{\mathrm {Gal}}}
\newcommand\SL{{\mathrm {SL}}}
\newcommand\Hom{{\mathrm {Hom}}}
\newtheorem{thm}{Theorem}
\newtheorem{theorem}[thm]{Theorem}
\newtheorem{cor}[thm]{Corollary}
\newtheorem{conj}[thm]{Conjecture}
\theoremstyle{proposition}
\newtheorem{prop}[thm]{Proposition}
\newtheorem{lemma}[thm]{Lemma}
\theoremstyle{definition}
\newtheorem{definition}[thm]{Definition}
\newtheorem{remark}[thm]{Remark}
\newtheorem{example}[thm]{Example}
\newtheorem{claim}[thm]{Claim}
\newtheorem{lem}[thm]{Lemma}
\theoremstyle{definition}
\newtheorem{dfn}{Definition}
\theoremstyle{remark}
\theoremstyle{remark}
\newtheorem*{fact}{Fact}
\author{ SriLakshmi Krishnamoorthy }
\address{ Srilakshmi Krishnamoorthy
 \newline
INDIAN INSTITUTE OF SCIENCE EDUCATION AND RESEARCH, THIRUVANANTHAPURAM, INDIA.}
\email{srilakshmi@iisertvm.ac.in}
\author{Sunil Kumar Pasupulati }
\address{ Sunil Kumar Pasupulati
\newline
\ \ INDIAN INSTITUTE OF SCIENCE EDUCATION AND RESEARCH, THIRUVANANTHAPURAM, INDIA.}
\email{sunil4960016@iisertvm.ac.in}

\title{
Note on the $p$-divisibility of class numbers of an infinite family of imaginary quadratic fields. }
\begin{abstract}
For any odd prime $p,$ we construct an infinite family  of  imaginary quadratic fields  whose class numbers are  divisible by $p$. We give a  corollary  
which  settles Iizuka's conjecture for the case
$n=1$ and $p >2.$
\end{abstract}
\maketitle
\section{introduction}
Let $K$ be a number field. The ideal class group $Cl_K$ is defined to be the quotient group $J_K / P_K,$ where $J_K$ is the group of fractional ideals of $K$ and $P_K$ is  the group of principal fractional ideals of $K.$
It is well known that $Cl_K$ is finite. The class number $h_K$ of a number field $K$ is the order of $Cl_K.$ 
The ideal class group is one of the most basic and mysterious objects in algebraic number theory. The class group has drawn the attention of several authors. The divisibility properties of the class number of number fields play a significant role in understanding the structure of the ideal class groups of number fields. 
Cohen-Lenstra heuristics are a set of conjectures about this structure. In general, given a positive integer $n,$ there is no characterization to find all quadratic fields whose class numbers are divisible by $n.$ For $n=3,$ these fields were characterized in \cite{KM00}.
  For a given integer $n >1,$ the Cohen-Lenstra heuristic \cite{CL} predicts that the proportion of imaginary quadratic fields with class number divisible by $n$ should be positive. It has been proved by several authors that for every $n >1,$ there exist infinitely many  quadratic  fields  whose class number is  divisible by $n$  ( \cite{YK09}, \cite{IA11}, \cite{HC02}, \cite{CHKP}, \citep{GR},\cite{SL}).\\

 
 B. H. Gross and D. E. Rohrlich proved that for any odd integer $n >3,$ 
  there are infinitely many imaginary quadratic fields ($\Q\left(\sqrt{1−4U^n}\right), U>1$ ) whose class numbers are divisible by $n.$  Furthermore  St\'{e}phane Louboutin  \cite{SL}  proved the same result by simplifying the Gross and Rohrlich’s proof and proved the following result on divisibility of class number of $\Q(\sqrt{1-4U^k})$ for $U>2.$ 

 \begin{theorem}\label{T2}
If $k\in \Z^{+}$ be odd number, then for any integer $U  \geq 2$ the ideal class groups of the imaginary
quadratic fields $\Q(\sqrt{1 − 4U^k})$ contain an element of order $k.$ 
\end{theorem}
  Murty \cite{RM98} proved that  the class number of $\Q\left(\sqrt{1−U^n}\right)$ is divisible by $n$ if $1-U^n$ is square-free. We study the divisibility of class number  of  families $\Q(\sqrt{1-2m^p})$ by all odd primes $p.$\\
  
The following result on the $3$-divisibility of the class number is proved by K. Chakraborty and A. Hoque (Theorem 3.2, \cite{CH19} ).  
\begin{theorem} 
The class number of  $\mathbb{Q}(\sqrt{1-2m^3})$ is divisible by $3$ for any odd integer $m > 1.$
\end{theorem}

We study the family $\mathbb{Q}(\sqrt{1-2m^p})$ for all odd primes $p$ and prime power $m=q^r, r \in \mathbb{N}.$ In the following theorem, we prove the $p$-divisibility of class numbers for this family by using the results of Yann Bugeaud and T. N. Shorey   \cite{SB}. 
\begin{theorem}\label{main theorem1}
For  prime numbers $p ,q \geq 3$ and $m = q^r, r \in \mathbb{N}$ such that $\mathbb{Q}(\sqrt{1-2m^p}) \neq \mathbb{Q}(\sqrt{-1}),$
the class number of  $\Q(\sqrt{1-2m^p})$  is divisible by $p.$
\end{theorem}

 The condition $\mathbb{Q}(\sqrt{1-2m^p}) \neq \mathbb{Q}(\sqrt{-1}),$  means that $2m^p-1$ is not a square. Siegel's theorem (Lemma \ref{st}) asserts that there are only finitely many $m\in \Z$  such that $2m^p-1$ is a square.\\
\\
The  Birch Swinnerton-Dyer conjecture is an elliptic curve analogue of the analytic class number formula.
For any elliptic curve defined over $\Q$ of rank zero and square-free conductor $N,$ if $p \mid | E(\Q)| ,$ under
certain conditions on the Shafarevich-Tate group $ {\Sha}_D,$ the first author   \cite{S16} showed that 
$p \mid |\Sha_D|$ if and only if   $p \mid h_K,$  where $K=\Q(\sqrt{-D}).$\\



\section{Iizuka's Conjecture} 
Y.Iizuka recently proves the following result on divisibility of the class numbers of imaginary quadratic fields in \cite{IY}.
\begin{theorem}\label{I}
There is an infinite family of pairs of imaginary quadratic fields
$\Q(\sqrt{d})\ and \  \Q(\sqrt{d+1})$ with $d\in \Z$ whose class numbers are both divisible by $3.$
\end{theorem} 
Based on the above theorem, Iizuka conjectured  the  following
\begin{conj}\label{conjecture}(Iizuka)
For any prime $p$ and any positive integer $n,$ there is an infinite family of $n+1$ successive real (or imaginary)
quadratic fields 
$$\Q(\sqrt{D}), \Q(\sqrt{D+1}), \cdots, \Q(\sqrt{D+n})$$ with $D \in \Z$ whose class numbers are divisible
by $p.$
\end{conj}
As a consequence of \Cref{T2} and \Cref{main theorem1}, we get a  generalization of   \Cref{I} for all odd prime numbers $p$ and prove the following corollary.
\begin{cor}\label{main theorem2}
For every odd  prime number  there is an infinite family of pairs of imaginary quadratic fields
$\Q(\sqrt{d})\ and \  \Q(\sqrt{d+1})$ with $d\in \Z$ whose class numbers are both divisible by $p.$
\end{cor}

{\bf{\it Proof of the Corollary \ref{main theorem2}.}}\\
Fix an odd prime $p.$
Consider the set
$$S_0 = \Big\{
 m \in \Z^+|\ \text{the class number of } \Q \left( \sqrt{1-2m^p}  \right)\ \text{is divisible by } p \Big\}.$$  By  Lemma \ref{st},
   the  equation  $1-2x^p  = -y^2$ has finitely many solutions $(x,y)\in \Z\times \Z$. 
   Hence it follows from \Cref{main theorem1} that  $S_0$  contains infinitely many odd prime powers, which implies that $S_0$ is an infinite set.
For $m\in S_0,$ the prime $p$ divides the class number of $$ \Q(\sqrt{4(1-2m^p)^p})=\Q(\sqrt{1-2m^p}).$$ 
Let $U=2m^p-1.$  Then $U\geq 2.$ Furthermore \Cref{T2} implies that $p$ divides the class number of $\Q(\sqrt{1-4U^p}).$ Now look at 

 $$\Q(\sqrt{1-4U^p})=\Q\left(\sqrt{1-4(2q^p-1)^p} \right)=\Q\left(\sqrt{4(1-2q^p)^p+1}\right).$$  Let $d=4(1-2m^p)^p.$ The prime $p$ divides class numbers of
 $\Q(\sqrt{d}),\Q(\sqrt{d+1}).$

 Now to conclude the corollary,  we need to prove the set 
  $\mathcal{A}=\Big\{\Bbb Q \left( \sqrt{1-2m^p} \right)| m \in S_0\Big\}$ is an infinite set. For every square-free  integer $d_0 \neq 0$,  let $f(x)= \frac{1-2x^p}{d_0}.$ The polynomial $f(x)$ has distinct roots in $\overline{\Q}.$ Thus by Lemma \ref{st}, the equation $y^2=f(x)$ has finitely many integral solutions. Hence  the infiniteness of $\mathcal{A}$  follows from that of $S_0$. 
\begin{remark}
The above corollary settles Iizuka's conjecture (\ref{conjecture}) for the case $n=1$ and $p >2.$ We found  a  similar result for a different families of imaginary quadratic fields in \cite{XC}.
J. Chattopadhyay and S. Muthukrishnan  \cite{JS19} answer a weaker version of Iizuka's conjecture for $p=3.$  
\end{remark}

\section{Preliminaries}
We recall some known results and prove some lemmas that are necessary for proving our main theorem.
\begin{definition}
Let $K$ be  a number field  and let $S$ be a finite set of valuations on $K$, containing all the archimedean valuations. Then $$ R_S=\{ \alpha \in K \ | \ \nu(\alpha) ≥ 0  \ for \ all \ \nu \not \in S\}$$
is called the set of $S$-integers.
\end{definition}
\begin{lemma}\label{st} ( Siegel's theorem, \cite{SJ}, Chapter IX, Theorem 4.3 ) Let $K$ be a number field and  S be a finite set of valuations on $K$, containing all the archimedean valuations. Let $f(X)\in K[X]$ be a polynomial of degree $d \geq 3$ with distinct roots in the algebraic closure $\overline{K}$ of K. Then the
equation $y^2 = f(x)$ has only finitely many solutions in $S$-integers $x,y \in R_S.$
\end{lemma}

 We recall some results of Yann Bugeaud and  T.N. Shorey  \cite{SB} on solutions of Diophantine equation  $D_1x^2+D_2=\lambda^2 k^y,$ where $D_1$ and $D_2$ are coprime positive integers, $k\geq 2$ is an integer coprime with $D_1D_2$ and $\lambda =\sqrt{2},2$  such that $\lambda =2$ if $k$ is even.

Let us denote $F_i$ to be the Fibonacci sequence defined by $F_0=0,F_1=1$ and  $F_{i}=F_{i-1}+F_{i-2}$ for all $i\geq 2.$ Let $L_i$ be the Lucas sequence defined by $L_0=2,L_1=1$ and satisfying $L_i=L_{i-1}+L_{i-2}$ for all $i\geq 2.$ Define the subsets $\sF, \sG, \sH$ of $\N \times \N \times \N$  by $$ \sF:=\bigr\lbrace (F_{i-2\epsilon},L_{i+\epsilon},F_i)\  |\  i\geq 2, \epsilon \in \{\pm1\}\bigr\rbrace,$$
$$\sG:= \bigr\lbrace (1,4k^r-1,k)\  |\  k\geq 2, r\geq 1 \bigr\rbrace, $$
\begin{align*}
\sH :=\bigr\lbrace (D_1,D_2,k)\  |\ \text{there exist positive integers $r$ and $s$ such that } \\ D_1s^2+D_2=\lambda^2k^r \ and\ 3D_1 s^2-D_2=\pm \lambda^2  \bigr\rbrace .
\end{align*}

Define $\sN (\lambda,D_1,D_2,p)$ to be the number of $(x,y)\in \mathbb{Z}^{+}\times \mathbb{Z}^{+}$ of the Diophantine equation  $D_1x^2+D_2= \lambda^2 p^y.$

\begin{theorem}\label{shorey}(\cite{SB}, Theorem 1)
Let p be a prime number. Then    we have $\sN (\lambda,D_1,D_2,p)\leq 1$ expect for $\sN(2  ,13,3,2) =\sN( \sqrt{2} , 7 , 11 , 3) = \sN (1 , 2 , 1 , 3) = \sN (2 , 7 , 1 , 2) = \sN ( \sqrt{2},1,1 ,5) = \sN ( \sqrt{2} , 1 , 1 , 13) = \sN (2 , 1 , 3 , 7) = 2$ and when $( D_1,D_2,p)$ belongs to one of the infinite families $\sF , \sG \ and \  \sH .$ 
\end{theorem}

\begin{lemma}\label{TL2}
 For any odd prime $q$ and any integer $D > 3,$  the equation $Dx^2+1=2q^y$ has at most one solution $(x,y)\in \mathbb{Z}^{+}\times \mathbb{Z}^{+}.$ 

\end{lemma}
\begin{proof}
Let $D_1=D$ and $D_2=1$ and $\lambda =\sqrt{2}$. We first note that $(\lambda,D_1,D_2,q) \not \in  \{(2,13,3,2),( \sqrt{2},7,11,3),$ 
$(1 , 2 , 1 , 3),(2 , 7 , 1 , 2),  ( \sqrt{2} , 1 , 1 , 5),  ( \sqrt{2} , 1 , 1 , 13), (2 , 1 , 3 , 7) \}. $ By Theorem \ref{shorey}, it is enough to show that   $(D_1,D_2,q) \not \in \sF.$ If  $(F_{i-\epsilon}, L_{i+\epsilon}, F_i) = (D_1, D_2, p),$ then  $i=2, \epsilon =-1$. Hence  $D=2.$ This is not possible because $D>3.$ Therefore $(D_1,D_2,q) \not \in  \sF.$ If $(1,4k^r-1,k) = (D_1, D_2, q)$, then $D=D_1 = 1.$ This is not possible. Therefore $(D_1,D_2,q) \not \in  \sG.$ If $D_1x^2 + 1 = 2q^y,$ then $3D_1 x^2 -1 \geq D_1 x^2 -1 \geq 2 q^y -2 \geq 4.$ Hence $(D_1,D_2,q) \not \in  \sH.$

\end{proof}
\begin{prop}\label{prop1}
 For  prime numbers $p ,q \geq 3$ and $m = q^r, r \in \mathbb{N},$ let $\alpha = 1+\sqrt{1-2m^p}, $ then $\pm2^\frac{p-1}{2} \alpha$ is not a  $p$-th power of an algebraic integer in $\Q(\sqrt{1-2m^p})$. 
\end{prop}
\begin{proof}
Let $d$ be the square-free   part of $\sqrt{1-2m^p}$ with signature.
 Let $K = \Q(\sqrt{d})$ and $\mathcal{O}_K$ be the ring of integers of $K.$
  Note that $-2^{\frac{p-1}{2}}\alpha$ is a $p$-th power in $\mathcal{O}_K$ if and only if
   $2^{\frac{p-1}{2}}\alpha$ is a $p$-th power in $\mathcal{O}_K.$ It is enough to show $2^{\frac{p-1}{2}}\alpha$ is not a $p$-th power. Suppose  that,
 \begin{align}\label{p1}
 2^{\frac{p-1}{2}}\alpha= \beta^p \  \mathrm{for} \ \mathrm{some} \ \beta = a+b\sqrt{d} \in  \mathcal{O}_K.
 \end{align}
  Then ,
 \begin{align} \label{p2}
 2^{\frac{p-1}{2}}\left( 1+\sqrt{1-2m^p}\right)=\sum_{j=0}^{\frac{p-1}{2}}{ p \choose 2j} a^{p-2j} b^{2j}d^j+\gamma \sqrt{d}  \  \mathrm{for} \ \mathrm{some}  \ \gamma \in \mathbb{Z}.
\end{align} 
By comparing constant terms on both sides, we have ,
\begin{align} \label{p3}
2^{\frac{p-1}{2}} =  \sum_{j=0}^{\frac{p-1}{2}}{ p \choose 2j} a^{p-2j} b^{2j}d^j.
\end{align} 
This implies that,
\begin{align*}
2^{\frac{p-1}{2}} =  a \Biggr(\sum_{j=0}^{\frac{p-1}{2}}{ p \choose 2j} a^{p-2j-1} b^{2j}d^j\Biggr).
\end{align*} 
Hence $a$ divides ${2^{\frac{p-1}{2}}}.$\\

{\textbf Case 1 :} $a$ is even.\\
We look at (\ref{p1})
\begin{align}\label{p5}
 {2^{\frac{p-1}{2}}}\alpha=\left(a+b\sqrt{d}\right)^p,
\end{align} 
applying the norm map on the both sides 
   $$(2m)^p= \left(a^2-b^2d\right)^p.$$ 
Hence we have  $$2m= a^2-b^2d.$$ 
Since $2\mid a$,  we obtain $2\mid b^2d.$ We deduce that  $2\mid b$ since $d$ is odd. Taking  divisibility of $a^2, b^2$ by $4$  into consideration, we conclude  that $4\mid {2m} $ but $m$ is odd which contradicts the assumption that  $a$ is even. \\
{\textbf Case 2 :} $a$ is odd. \\
Suppose that $a$ is odd. Since $a\mid {2^{\frac{p-1}{2}}}$, this implies that $a=\pm 1$.
Putting in \Cref{p1} we have 
$$2^{\frac{p-1}{2}}( 1+\sqrt{1-2m^p)}= \left(\pm 1+b\sqrt{d}\right)^p,$$ 
applying the norm map on both sides we get  
$$(2m)^p = {\left(1-b^2d\right)^p}.$$ 
Rewriting the above equation using $D=-d$ and $m=q^r$, we get
\begin{align}\label{p6}
1+Db^2=2q^r.
\end{align} 
We observe that  $1-2m^p=1-2(q^r)^p=b'^2d$ for some $b'\in \Z.$ Rephrasing this equation we have,
\begin{align}\label{p7}
1+b'^2D=2q^{rp}.
\end{align} 
Thus, from Equations \eqref{p6} and \eqref{p7}, we get $(b,r),(b',rp)$ are solutions of the equation \\
$Dx^2+1=q^y,$
which is a contradiction to \Cref{TL2}. Hence $\pm2^\frac{p-1}{2} \alpha$ is not a $p$-th power in $\mathcal{O}_K.$

\end{proof}

\section{Proof of the theorem}
 We now prove the main theorem of this article.\\
{\it Proof of \Cref{main theorem1}}.
 Let $d$ be the  square-free part of  $ 1 - 2m^p$ with signature then $d \equiv 3 \pmod 4$ and $K = \Q(\sqrt{d}).$
Put $\alpha:=1+\sqrt{1-2m^p}$, then $N_{K/\Q}(\alpha)=2m^p.$  Since  $d\equiv 3 \pmod 4$,  the ideal $(2)$ is ramified,  there exists a prime ideal  $\cP$ such that  $(2)=\cP^2.$ 
Since the norm of $\alpha $  is $2m^p=2q^{rp},$   the prime decomposition of $(\alpha)$ is given by 
$(\alpha)=\cP\cQ^t,$ for some positive integer $t,$  where $\cQ$ is a prime which lies above $q.$ 
Then $N_{K/\Q}((\alpha))=2 q^t$, where $N(\cQ)=q$ (since $q$ splits in $\Q(\sqrt{d})$ as $\left(\frac{d}{q}\right)=1).$ Hence $ t = rp.$\\
Consider the ideal  $I:=\cP\cQ^{\frac{t}{p}},$  of $K.$ Observe that
$$I^p= \cP^p\cQ^t= (2)^{\frac{p-1}{2}}\cP\cQ^t =(2)^{\frac{p-1}{2}} (\alpha)=(2^{\frac{p-1}{2}} \alpha).$$ 
We claim that the  order of the ideal $I$ in ideal class group is $p.$ Suppose not, let $(\beta)=I$ for some $\beta$ in $\mathcal{O}_K.$  Then 
$$(\beta^p)=(\beta)^p= I^p = (2^{\frac{p-1}{2}} \alpha).$$
Since the only units of $\sO_K$ are $\{ 1,-1\},$ this implies that  $I$  is  a principal ideal if and only if  $\pm 2^\frac{p-1}{2} \alpha$  is a power of $p$ in $\sO_K .$  From Proposition \ref{prop1}, we know that  $\pm 2^\frac{p-1}{2} \alpha$ is not a $p$-th power in $\sO_K .$  Hence the class group of $\Q(\sqrt{1-2m^p}) $  has an  element $I$ of order $p.$  \hfill $\square$

We prove a corollary of Theorem \ref{main theorem1}.
\begin{cor}

For every odd prime $p$, there exist infinitely many imaginary biquadratic fields whose class number is divisible by $p$. 

\end{cor}
\begin{proof}
Fix an odd prime $p$. Consider the set 
\begin{align*}
S_1=\{ m\in \Z^+|\ m \ \text{ is not a square,}\ & m\equiv 1\pmod 4 \ \text{and } \\ &\text{the class number of}\ \Q(\sqrt{1-2m^p})\ \text{is divisible by }p\}.
\end{align*} 

By Lemma \ref{st}, the equation $1-2x^p=y^2$ has only finitely many solutions $(x,y) \in \Z\times \Z.$ Hence it follows from Theorem 3 and Dirichlet's theorem on arithmetic progression that $S_1$ contains infinitely many primes $q$ with $q\equiv 1\pmod{4},$ which implies that $S_1$ is an infinite set.

For $m \in S_1$, consider the imaginary biquadratic field $K_m=\Q(\sqrt{1-2m^p},\sqrt{m}).$ 
 Denote $L_{m}^1 :=\Q(\sqrt{1-2m^p}), L^2_{m} :=\Q(\sqrt{m})$ and $L_{m}^3:=\Q\left(\sqrt{1-2m^p} \sqrt{m}\right).$
  Since $m$ is not a square, $L^2_m$ is actually a quadratic field. We observe that $L_m^1\neq L_m^2$  because $1-2m^p\equiv 3 \pmod{4}.$ Thus $L_{m}^1,  L^2_{m}$  and $L_{m}^3$ are the three quadratic subfields of $K_m.$
 Let $h_{m}, h_{m}^1, h_{m}^2$ and  $h_{m}^3$ be the class numbers of $K_{m}, L_{m}^1, L_{m}^2$ and $L_{m}^3$ respectively. Then by Lemma 2 in \cite{GC91}, we have $h_{m}= \frac{h_{m}^1 h_{m}^2 h_{m}^3} {2^i}$  where $i= 0,1.$ 
 Since $m \in S_1,$ the prime $p$ divides $h_{m}^1.$ Since $p$ is odd, $p$ divides $h_m.$ The infiniteness of the set $\{ K_m| m\in S_1\}$ follows from that of the set  $S_1.$

\end{proof}

\section*{Acknowledgements}
We thank Prof. Kalyan Chakraborty for introducing this problem to us. We thank IISER Tiruvananthapuram for providing excellent working conditions. We thank  Azizul Hoque for his comments and the careful reading of this article. We acknowledge SageMath for the numerical evidence of our theorems. We like to express our deep gratitude to the anonymous referee for carefully analyzing the article,  refining several arguments and making suggestions for a better presentation.
\bibliographystyle{amsplain}
\bibliography{p-divisibility-class-numbers-updated.bib}
\end{document}